\documentclass[11pt,leqno]{amsart}

\usepackage{amsthm}
\usepackage{amsmath}
\usepackage{amssymb}
\usepackage[all,cmtip]{xy}
\usepackage[T1]{fontenc}
\usepackage{mathrsfs}
\usepackage{enumerate}
\usepackage{amscd}

\theoremstyle{plain}
\newtheorem{thm}{Theorem}

\newtheorem{prop}[thm]{Proposition}
\newtheorem{cor}[thm]{Corollary}
\theoremstyle{definition}
\newtheorem{defn}[thm]{Definition}

\newtheorem{exmp}[thm]{Example}
\newtheorem{rem}[thm]{Remark}

\newtheorem{que}[thm]{Question}

\title{Noncommutatively Graded Algebras}

\date{}

\author{Patrik Nystedt}

\address{University West,
Department of Engineering Science, 
SE-46186 Trollh\"{a}ttan, Sweden}

\email{Patrik.Nystedt@hv.se}



\begin{document}

\maketitle

\begin{abstract}
Inspired by the commutator and anticommutator 
algebras derived from algebras graded by groups, we introduce 
noncommutatively graded algebras.
We generalize various classical graded results 
to the noncommutatively graded situation
concerning identity elements, inverses,
existence of limits and colimits and
adjointness of certain functors.
In the particular instance of noncommutati\-vely
graded Lie algebras, we establish the existence
of universal graded enveloping algebras and we show a 
graded version of the Poincar\'{e}-Birkhoff-Witt theorem.
\end{abstract}

\section{Introduction}

Let $R$ denote a ring which is unital, commutative and associative.
Suppose that $A$ is an $R$-algebra. By this we mean that $A$ is 
a left $R$-module equipped with an $R$-bilinear map 
$A \times A \ni (a,b) \mapsto ab \in A$.
Let $G$ be a group. 

Recall that $A$ is called $G$-graded if there is a family 
$\{ A_g \}_{g \in G}$ of $R$-submodules of $A$ such that
\begin{equation}\label{directsum}
A = \oplus_{g \in G} A_g,
\end{equation}
as $R$-modules, and for all 
$g,h \in G$ the inclusion
\begin{equation}\label{graded} 
A_g A_h \subseteq A_{gh}
\end{equation}
holds. 
The category of $G$-graded $R$-algebras,
here denoted by $G$-GA, is obtained by taking
$G$-graded $R$-algebras as objects and
for the morphisms between such objects
$A$ and $B$ we take the $R$-algebra homomorphisms
$f : A \rightarrow B$ satisfying %such that for every $g \in G$ 
$f(A_g) \subseteq B_g$, %holds.
for $g \in G$.

Graded algebras include as special cases many other constructions
such as polynomial and skew polynomial rings, 
Ore extensions,
matrix rings, 
Morita contexts,
group rings,
twisted group rings,
skew group rings and
crossed products.
Therefore, the theory of graded algebras
not only gives new results for several constructions
simultaneously, but also serves as a unification of known theorems.
For more details concerning graded algebras,
see e.g. \cite{nas1982} or \cite{nas2004} and the references therein.

The motivation for the present article is the observation
that there are many natural examples of algebras which 
satisfy (\ref{directsum}) but only a weaker 
form of (\ref{graded}).
Namely, suppose that
$A$ is an associative $G$-graded algebra. Then the induced Lie algebra
$(A , [ \cdot , \cdot ] )$ 
(see \cite[p. 3]{schafer1995})
and the induced Jordan algebra 
$(A , \{ \cdot , \cdot \} )$
(see \cite[p. 4]{schafer1995}, where 
\begin{equation}\label{defcommutator}
A \times A \ni (a,b) \mapsto [a,b] = ab-ba \in A
\end{equation}
is the commutator and
$$A \times A \ni (a,b) \mapsto \{ a, b \} = ab+ba \in A$$ 
is the anticommutator, satisfy
\begin{equation}\label{commutator}
[ A_g , A_h ] \subseteq A_{gh} + A_{hg}
\end{equation}
and
\begin{equation}\label{anticommutator}
\{ A_g , A_h \} \subseteq A_{gh} + A_{hg}
\end{equation}
respectively, for all $g,h \in G$
(see also Example \ref{lambdamu}).
Inspired by (\ref{commutator}) and (\ref{anticommutator}) we say that
an $R$-algebra $A$ is {\it noncommutatively $G$-graded} if there is a family 
$\{ A_g \}_{g \in G}$ of $R$-submodules of $A$ satisfying
(\ref{directsum}), as $R$-modules, and for all $g,h \in G$ the inclusion
$A_g A_h \subseteq A_{gh} + A_{hg}$
holds.
The aim of this article is to generalize various $G$-graded classical results
to the noncommutatively $G$-graded situation.
Here is an outline of the article.

In Section \ref{gradedalgebras}, we introduce the category
$G$-NCGA of noncommutatively $G$-graded algebras
(see Definition \ref{definitionnoncommutativelygraded}) and we show 
results concerning identity elements, inverses,
existence of limits and colimits and
adjointness of certain functors related to $G$-NCGA (see Propositions \ref{unit}-\ref{adjointagain}).
In Section \ref{gradedtensoralgebra}, we fix the notation
concerning $G$-graded modules and we recall the construction
of the graded tensor algebra
(see Propositions \ref{gradingdirectsum}-\ref{tensor}). 
In Section \ref{gradedliealgebras}, we study the
particular instance of noncommutatively $G$-graded Lie algebras
(see Definition \ref{gradedliealgebras}),
we establish the existence of universal graded enveloping algebras 
(see Proposition \ref{lieenveloping}) and we show a 
graded version of the Poincar\'{e}-Birkhoff-Witt theorem
(see Proposition \ref{pbw}).

\section{Graded algebras}\label{gradedalgebras}

In this section, we introduce the category of
$G$-NCGA of noncommutatively $G$-graded algebras
(see Definition \ref{definitionnoncommutativelygraded}) and we show 
results concerning identity elements, inverses,
existence of limits and colimits and
adjointness of certain functors related to $G$-NCGA (see Propositions \ref{unit}-\ref{adjoint}).

For the rest of this article, $R$ denotes 
a unital, commutative and associative ring,
$A$ denotes an $R$-algebra
and $G$ denotes a multiplicatively written group 
with identity element $e$.
If $A$ is unital, then we let $1$ denote the
multiplicative identity of $A$.

\begin{defn}\label{definitionnoncommutativelygraded}
We say that $A$ is noncommutatively $G$-graded if there is a family 
$\{ A_g \}_{g \in G}$ of $R$-submodules of $A$
such that $A = \oplus_{g \in G} A_g$, as $R$-modules, and 
for all $g,h \in G$ the inclusion $A_g A_h \subseteq A_{gh} + A_{hg}$ holds.
The category of noncommutatively $G$-graded algebras,
denoted by $G$-NCGA, is obtained by taking
noncommutatively $G$-graded algebras as objects 
and for the morphisms between such objects
$A$ and $B$ we take the $R$-algebra homomorphisms
$f : A \rightarrow B$ satisfying
$f(A_g) \subseteq B_g$, for $g \in G$. 
\end{defn}

\begin{rem}
Clearly, there is an inclusion of $G$-GA in $G$-NCGA.
If $G$ is abelian, then $G$-NCGA coincides with $G$-GA.
\end{rem}

\begin{exmp}\label{lambdamu}
Let $A$ be a $G$-graded algebra.
From $A$ it is easy to construct many examples of noncommutatively $G$-graded algebras.
Indeed, for all $g,h \in G$, take $\lambda_{g,h} , \mu_{h,g} \in R$.
Define a new product $\bullet$ on $A$ by the additive extension
of the relations $a_g \bullet b_h = \lambda_{g,h} a_g a_h + \mu_{h,g} a_h a_g,$
for $g,h \in G$, $a_g \in A_g$ and $b_h \in B_h$.
We will denote this algebra by $(A , \lambda , \mu)$. 
Note that if we take $\lambda \equiv 1$ and $\mu \equiv -1$ (or $\mu \equiv 1$),
then $(A , \lambda , \mu)$ coincides with the commutator 
(or anticoommutator) algebra defined by $A$. 
The category $G$-NCGA$_{\lambda,\mu}$ is obtained by taking
$(A,\lambda,\mu)$, for $G$-graded algebras $A$, as objects,
and for morphisms $(A,\lambda,\mu) \rightarrow (B,\lambda,\mu)$
we take graded $R$-algebra morphisms $A \rightarrow B$.
It is clear that the correspondence $A \mapsto (A,\lambda,\mu)$, on objects of $G$-NCGA,
and by the identity, on graded $R$-algebra morphisms, defines a functor 
from $G$-GA to $G$-NCGA$_{\lambda,\mu}$.
We will denote this functor by $(\lambda,\mu)$.
\end{exmp}

\begin{rem}
Suppose that $A$ is a unital $G$-graded algebra.
Then, from \cite[Proposition 1.1.1.1]{nas2004} it follows that
$1 \in A_e$. If for every $g \in G$, the conditions 
$\lambda_{g,e} + \mu_{e,g} = 1 = \lambda_{e,g} + \mu_{g,e}$ hold,
then $(A,\lambda,\mu)$ is also unital with multiplicative identity 1.
If $2$ is invertible in $R$, then this holds if $\lambda \equiv \mu \equiv 2^{-1}$.
\end{rem}

The next result is a generalization of \cite[Proposition 1.1.1.1]{nas2004}.

\begin{prop}\label{unit}
If $A$ is unital and noncommutatively $G$-graded, 
then $1 \in A_e$.
\end{prop}

\begin{proof}
Suppose that $1 = \sum_{g \in G} a_g$ for some
$a_g \in R_g$ satisfying $a_g = 0$ for all but finitely many $g \in G$.
Take $h \in G$.
Then 
$$a_h = 1 a_h = \sum_{g \in G} a_g a_h =
a_e a_h + \sum_{g \in G \setminus \{ e\} } a_g a_h.$$
Thus 
$$A_h \ni a_h - a_e a_h =  \sum_{g \in G \setminus \{ e \} } a_g a_h \in 
\sum_{g \in G \setminus \{ e \}} ( A_{gh} + A_{hg} ) \subseteq
\oplus_{g \in G \setminus \{ h \}} A_g.$$
Hence, in particular, we get that %$r_h = r_e r_h$ and
\begin{equation}\label{unitfirst}
\sum_{g \in G \setminus \{ e \} } a_g a_h = 0.
\end{equation}
Also 
$$a_h = a_h 1 = \sum_{g \in G} a_h a_g =
a_h a_e + \sum_{g \in G \setminus \{ e\} } a_h a_g.$$
Thus 
$$A_h \ni a_h - a_h a_e =  \sum_{g \in G \setminus \{ e \} } a_h a_g \in 
\sum_{g \in G \setminus \{ e \}} ( A_{gh} + A_{hg} ) \subseteq
\oplus_{g \in G \setminus \{ h \}} A_g.$$
Hence, in particular, we get that
\begin{equation}\label{unitsecond}
a_h = a_h a_e.
\end{equation}
From (\ref{unitfirst}) and (\ref{unitsecond}) it follows
that $0 = \sum_{g \in G \setminus \{ e \} } a_g a_e = 
\sum_{g \in G \setminus \{ e \} } a_g$.
Thus $$1 = \sum_{g \in G} a_g = a_e + \sum_{g \in G \setminus \{ e \} } a_g =
a_e + 0 = a_e \in A_e.$$
\end{proof}

If $A$ is noncommutatively $G$-graded,
then $\cup_{g \in G} A_g$ is called the set of homogeneous
elements of $A$; if $g \in G$, then a nonzero element $a \in A_g$
is said to be homogeneous of degree $g$.
In that case we write ${\rm deg}(a) = g$.
The following result is a generalization of \cite[Proposition 1.1.1.2]{nas2004}

\begin{prop}\label{inverse}
Suppose that $A$ is unital and noncommutatively $G$-graded.
If $a$ is a non-zero homogeneous element of $A$ such that ${\rm deg}(a)=g$
and $a$ has a right (left) inverse,
then $a$ has a right (left) homogeneous inverse
of degree $g^{-1}$.
\end{prop}

\begin{proof}
First show the ''right'' part of the proof.
Suppose that there is $b \in A$ such that $ab = 1$.
Suppose that $b = \sum_{h \in G} b_h$ for some $b_h \in A_h$
such that $b_h = 0$ for all but finitely many $h \in G$.
Then 
$$1 = ab = \sum_{h \in G} a b_h = 
a b_{g^{-1}} + \sum_{h \in G \setminus \{ g^{-1} \} } a b_h.$$
Thus, from Proposition \ref{unit}, we get that
$$A_e \ni 1 - a b_{g^{-1}} = \sum_{h \in G \setminus \{ g^{-1} \} } a b_h \in 
\sum_{ h \in G \setminus \{ g^{-1} \} } ( A_{gh} + A_{hg} ) \subseteq
\oplus_{g \in G \setminus \{ e \}} A_g.$$
Therefore $1 - a b_{g^{-1}} = 0$ and thus $a b_{g^{-1}} = 1$.

Now we show the ''left'' part of the proof.
Suppose that there is $c \in A$ such that $ca = 1$.
Suppose that $c = \sum_{h \in G} c_h$ for some $c_h \in A_h$
such that $c_h = 0$ for all but finitely many $h \in G$.
Then 
$$1 = ca = \sum_{h \in G} c_h a = 
c_{g^{-1}} a + \sum_{h \in G \setminus \{ g^{-1} \} } c_h a.$$
Thus, from Proposition \ref{unit}, we get that
$$A_e \ni 1 - c_{g^{-1}} a = \sum_{h \in G \setminus \{ g^{-1} \} } c_h a \in 
\sum_{ h \in G \setminus \{ g^{-1} \} } ( A_{hg} + A_{gh} ) \subseteq
\oplus_{g \in G \setminus \{ e \}} A_g.$$
Therefore $1 - c_{g^{-1}} a = 0$ and thus $c_{g^{-1}} a = 1$.
\end{proof}

\begin{prop}
Inverse limits exist in $G$-{\rm NCGA}.
\end{prop}

\begin{proof}
Suppose that we are given a preordered set $(I,\leq)$ and
a family of noncommutatively $G$-graded $R$-algebras
$( A_{\alpha} )_{\alpha \in I}$.
For all $\alpha,\beta \in I$ with $\alpha \leq \beta$
let $f_{\alpha \beta} : A_{\beta} \rightarrow A_{\alpha}$
be a morphism in $G$-NCGA.
Suppose that the the morphisms $f_{\alpha \beta}$ form an inverse system,
that is, that the following conditions hold:
\begin{itemize}
\item[(i)] the relations $\alpha \leq \beta \leq \gamma$
imply that $f_{\alpha \gamma} = f_{\alpha \beta} \circ f_{\beta \gamma}$;
\item[(ii)] for every $\alpha \in I$, the equality 
$f_{\alpha \alpha} = {\rm id}_{A_{\alpha}}$ holds.
\end{itemize}
Let $P$ denote the product of the sets $A_{\alpha}$
and let $p_{\alpha} : P \rightarrow A_{\alpha}$
denote the corresponding projection.
Let $Q$ denote the subset of all $x \in P$
which satisfy $p_{\alpha}(x) = f_{\alpha \beta} ( p_{\beta}(x) )$
for all $\alpha,\beta \in I$ such that $\alpha \leq \beta$.
Take $r \in R$ and $x,y \in Q$.
Put $r x = ( r p_{\alpha}(x) )_{\alpha \in I}$,
$x + y = ( p_{\alpha}(x) + p_{\alpha}(y) )_{\alpha \in I}$ and
$x y = ( p_{\alpha}(x) p_{\alpha}(y) )_{\alpha \in I}$.
From general results concerning inverse limits of magmas
with operations (see \cite[\S 10]{bourbaki1974}) we know that
this defines a well defined $R$-algebra structure on $Q$
making it an inverse limit in the category of $R$-algebras.
Take $g,h \in G$ and let $Q_g'$ denote all
$x \in Q$ such that for each $\alpha \in I$, the
relation $p_{\alpha}(x) \in (A_{\alpha})_g$ holds.
Put $Q' = \oplus_{g \in G} Q_g'$.
Now we show that $Q'$ is noncommutatively graded.
Take $x \in Q_g'$ and $y \in Q_h'$.
Take $\alpha \in I$. Then 
$p_{\alpha} (xy) = p_{\alpha}(x) p_{\alpha}(y) \in 
( A_{\alpha} )_g ( A_{\alpha} )_h \subseteq 
( A_{\alpha} )_{gh} + ( A_{\alpha} )_{hg}$.
Therefore $Q_g' Q_h' \subseteq Q_{gh}' + Q_{hg}'$.
Now we show that $Q'$ is an inverse limit in the category $G$-NCGA.
For each $\alpha \in I$, let $f_{\alpha}$
denote the map of noncommutatively graded algebras $Q' \rightarrow A_{\alpha}$
defined by restriction of $p_{\alpha}$
and suppose that $u_{\alpha} : F \rightarrow A_{\alpha}$ is a 
graded map for some $G$-noncommutatively graded $R$-algebra
$F$ into $A_{\alpha}$ such that $f_{\alpha \beta} \circ u_{\beta} = u_{\alpha}$
whenever $\alpha \leq \beta$.
Then there exists a unique graded map $u : F \rightarrow Q'$
such that $u_{\alpha} = f_{\alpha} \circ u$ for all $\alpha \in I$.
First we show uniqueness of $u$.
Take $y \in F_g$.
From the relations $u_{\alpha} = f_{\alpha} \circ u$, for $\alpha \in I$,
it follows that $u(y) = ( u_{\alpha}( y ) )_{\alpha \in I}$.
Next we show that $u$ is a well defined morphism in $G$-NCGA.
To this end, suppose that $g \in G$, $y \in F_g$ and $\alpha \leq \beta$.
Then $( f_{\alpha \beta} \circ p_{\beta} ) (u (y) ) =
( f_{\alpha \beta} \circ u_{\beta} ) (y) = u_{\alpha} (y) = p_{\alpha} (u(y))$.
Therefore $u(y) \in Q$. Since $y \in F_g$ we get that 
$u_{\alpha}(y) \in ( A_{\alpha} )_g$ and thus $u(y) \in Q_g'$.
Since each $u_{\alpha}$ is an
$R$-algebra homomorphism, the same holds for $u$. 
\end{proof}

\begin{prop}
Direct limits exist in $G$-{\rm NCGA}.
\end{prop}

\begin{proof}
Suppose that we are given a directed set $(I,\leq)$ and
a family of noncommutatively $G$-graded $R$-algebras
$( A_{\alpha} )_{\alpha \in I}$.
For all $\alpha,\beta \in I$ with $\alpha \leq \beta$
let $f_{\beta \alpha} : A_{\alpha} \rightarrow A_{\beta}$
be a morphism in $G$-NCGA.
Suppose that the the morphisms $f_{\beta \alpha}$ form a direct system,
that is, that the following conditions hold:
\begin{itemize}
\item[(i)] the relations $\alpha \leq \beta \leq \gamma$
imply that $f_{\gamma \alpha} = f_{\gamma \beta} \circ f_{\beta \alpha}$;
\item[(ii)] for every $\alpha \in I$, the equality 
$f_{\alpha \alpha} = {\rm id}_{A_{\alpha}}$ holds.
\end{itemize}
Take $g \in G$.
Let $D_g$ denote the direct sum of the sets $( ( A_{\alpha} )_g )_{\alpha \in I}$.
We will identify each $(A_{\alpha})_g$ with its image in $D_g$.
For each $x \in D_g$ let $\lambda_g(x)$ denote the unique element in $I$
such that $x \in ( A_{\lambda_g(x)} )_g$.
Define an equivalence relation $\sim_g$ on $D_g$ by saying that
if $x,y \in D_g$, then $x \sim_g y$ whenever there is $\gamma \in I$
with $\lambda_g(x) \leq \gamma$, $\lambda_g(y) \leq \gamma$ and
$f_{\gamma \lambda_g(x)}(x) = f_{\gamma \lambda_g(y)}(y)$.
Put $C_g = D_g / \sim_g$.
Denote by $(f_{\alpha})_g$ the restriction to $(A_{\alpha})_g$ of the
canonical mapping $f_g$ of $D_g$ onto $C_g$.
Denote by $(f_{\beta\alpha})_g$ the restriction of
$f_{\beta\alpha}$ to $( A_{\alpha} )_g$.
Then it follows that $(f_{\beta})_g \circ ( f_{\beta \alpha} )_g = ( f_{\alpha} )_g$
for $\alpha \leq \beta$.
Then $C_g$ is the direct limit of the $R$-modules $(A_{\alpha})_g$
with a well defined $R$-module structure defined as follows.
Take $r \in R$ and $x,y \in C_g$.
There is $\alpha \in I$ and $x_{\alpha},y_{\alpha} \in D_g$ such that
$x = (f_{\alpha})_g (x_{\alpha})$ and
$y = (f_{\alpha})_g (y_{\alpha})$.
Put $rx = (f_{\alpha})_g (r x_{\alpha})$ and
$x + y = (f_{\alpha})_g (x_{\alpha} + y_{\alpha})$
(for details, see \cite[\S 10]{bourbaki1974}).
Put $C = \oplus_{g \in G} C_g$ (external direct sum of $R$-modules).
Now we will define a multiplication on $C$.
By additivity it is enough to define this on graded components.
Take $g,h \in G$, $x \in C_g$ and $y \in C_h$.
There is $\alpha \in I$, 
$x_{\alpha} \in D_g$ and $y_{\alpha} \in D_h$ such that 
$x = (f_{\alpha})_g (x_{\alpha})$ and
$y = (f_{\alpha})_h (y_{\alpha})$.

Case 1: $gh=hg$. Then put $xy = (f_{\alpha})_{gh}( x_{\alpha} y_{\alpha} )$.

Case 2: $gh \neq hg$. Then put
$xy = ( (f_{\alpha})_{gh} \circ p_{gh} + 
(f_{\alpha})_{hg} \circ p_{hg} ) ( x_{\alpha} y_{\alpha} )$.
Here $p_{gh} : D_{gh} \oplus D_{hg} \rightarrow D_{gh}$ 
and $p_{hg} : D_{gh} \oplus D_{hg} \rightarrow D_{hg}$ 
denote the corresponding projections.

From the definition of this multiplication, it follows that it
is $R$-bilinear and that $C$ is noncommutatively $G$-graded.
Now we show that it is well defined.
To this end, suppose that we take $\beta \in I$, 
$x_{\beta}' \in D_g$ and $y_{\beta}' \in D_h$ such that 
$x = (f_{\beta})_g (x_{\beta}')$ and
$y = (f_{\beta})_h (y_{\beta}')$.
By the definition of the direct limit there exists $\gamma \in I$
such that $\alpha \leq \gamma$, $\beta \leq \gamma$,
$x_{\gamma} := f_{\gamma \alpha} (x_{\alpha}) =
x_{\gamma}' := f_{\gamma \beta} (x_{\beta}')$ and
$y_{\gamma} := f_{\gamma \alpha} (y_{\alpha}) =
y_{\gamma}' := f_{\gamma \beta} (y_{\beta}')$.

Case 1: $gh=hg$. Then 
$$(f_{\alpha})_{gh}( x_{\alpha} y_{\alpha} ) =
(f_{\gamma})_{gh} ( f_{\gamma\alpha} ( x_{\alpha} y_{\alpha} ) ) =
( f_{\gamma} )_{gh} ( f_{\gamma\alpha}(x_{\alpha}) f_{\gamma\alpha}(y_{\alpha}) ) $$
$$= (f_{\gamma})_{gh} ( x_{\gamma} y_{\gamma} ) =
(f_{\gamma})_{gh} ( x_{\gamma}' y_{\gamma}' ) =
( f_{\gamma} )_{gh} ( f_{\gamma\beta}(x_{\beta}') f_{\gamma\beta}(y_{\beta}') ) $$
$$= ( f_{\gamma} )_{gh} ( f_{\gamma\beta} (x_{\beta}' y_{\beta}') ) =
(f_{\beta})_{gh}( x_{\beta}' y_{\beta}' ).$$

Case 2: $gh \neq hg$. Then
$$( (f_{\alpha})_{gh} \circ p_{gh} + 
(f_{\alpha})_{hg} \circ p_{hg} ) ( x_{\alpha} y_{\alpha} ) =$$
$$= (f_{\gamma})_{gh} ( f_{\gamma\alpha} ( p_{gh} ( x_{\alpha} y_{\alpha} ) ) ) +
(f_{\gamma})_{gh} ( f_{\gamma\alpha} ( p_{hg} ( x_{\alpha} y_{\alpha} ) ) ) $$
$$= [\mbox{$f_{\gamma\alpha}$ graded}] = 
(f_{\gamma})_{gh} ( p_{gh} ( f_{\gamma\alpha} ( x_{\alpha} y_{\alpha} ) ) ) +
(f_{\gamma})_{gh} ( p_{hg} ( f_{\gamma\alpha} ( x_{\alpha} y_{\alpha} ) ) ) $$
$$= (f_{\gamma})_{gh} ( p_{gh} ( 
f_{\gamma\alpha} ( x_{\alpha} ) f_{\gamma\alpha} ( y_{\alpha} ) ) ) +
(f_{\gamma})_{gh} ( p_{hg} ( 
f_{\gamma\alpha} ( x_{\alpha} ) f_{\gamma\alpha} y_{\alpha} ) ) ) $$
$$= (f_{\gamma})_{gh} ( p_{gh} ( x_{\gamma}' y_{\gamma}' ) ) +
(f_{\gamma})_{hg} ( p_{hg} ( x_{\gamma}' y_{\gamma}' ) )$$
$$= ( f_{\gamma} )_{gh} ( p_{hg} ( f_{\gamma\beta}(x_{\beta}') f_{\gamma\beta}(y_{\beta}') ) )
+  ( f_{\gamma} )_{hg} ( p_{hg} ( f_{\gamma\beta}(x_{\beta}') f_{\gamma\beta}(y_{\beta}') ) )$$
$$= ( f_{\gamma} )_{gh} ( p_{hg} ( f_{\gamma\beta}(x_{\beta}' y_{\beta}') ) )
+  ( f_{\gamma} )_{hg} ( p_{hg} ( f_{\gamma\beta}(x_{\beta}' y_{\beta}') ) )$$
$$= [ \mbox{$f_{\gamma\beta}$ graded} ] = 
( f_{\gamma} )_{gh} ( f_{\gamma\beta} ( p_{gh} (x_{\beta}' y_{\beta}') ) )
+  ( f_{\gamma} )_{hg} ( f_{\gamma\beta} ( p_{hg} (x_{\beta}' y_{\beta}') ) )$$
$$= ( (f_{\beta})_{gh} \circ p_{gh} + 
(f_{\beta})_{hg} \circ p_{hg} ) ( x_{\beta}' y_{\beta}' ).$$ 
Now we show that $C$ is a direct limit of the $A_{\alpha}$
and the maps $f_{\beta \alpha}$.
Suppose that $\alpha, \beta \in I$ satisfy $\alpha \leq \beta$.
Define $f_{\alpha} : A_{\alpha} \rightarrow C$ in the following way.
Take $a \in A_{\alpha}$. Then $a = \sum_{g \in G} a_g$
for some $a_g \in (A_{\alpha})_g$ such that $a_g = 0$
for all but finitely many $g \in G$.
Put $f_{\alpha}(a) = \sum_{g \in G} (f_{\alpha})_g(a_g)$.
From the fact that (see above) 
$(f_{\beta})_g \circ ( f_{\beta \alpha} )_g = ( f_{\alpha} )_g$,
for $g \in G$, it follows that 
$f_{\beta} \circ f_{\beta \alpha}  = f_{\alpha}$.
Suppose that $F$ is a noncommutatively $G$-graded algebra
and that there are graded algebra maps $u_{\alpha} : A_{\alpha} \rightarrow F$, 
for $\alpha \in I$, satisfying $u_{\beta} \circ f_{\beta\alpha} = u_{\alpha}$,
whenever $\alpha \leq \beta$.
Then there is a unique graded algebra map 
$u : C \rightarrow F$ such that 
$u \circ f_{\alpha} = u_{\alpha}$ for $\alpha \in I$.
In fact, take $g \in G$ and $c_g \in C_g$.
Then there is $a_{\alpha,g} \in A_g$ such that 
$(f_{\alpha})_g(a_{\alpha,g}) = c_g$.
Put $u(c_g) = u_{\alpha} ( a_{\alpha,g} )$.
Then $u \circ f_{\alpha} = u_{\alpha}$ and it is clear from these 
relations that $u$ has to be defined in this way. Thus uniqueness of $u$ follows. 
Now we show that $u$ is well defined.
Suppose that $\alpha \leq \beta$ and put $a_{\beta,g} = f_{\beta\alpha}(a_{\alpha,g})$.
Then $u_{\beta}( a_{\beta,g} ) = 
u_{\beta}( f_{\beta\alpha}(a_{\alpha,g}) ) =
u ( f_{\beta} ( f_{\beta\alpha} ( a_{\alpha,g} ) ) )=
u ( f_{\alpha} ( a_{\alpha,g} ))= 
u_{\alpha} ( a_{\alpha,g} )$.
\end{proof}

Recall that $A^{\rm op}$ equals the algebra $A$
as a left $R$-module, but with a new product $\cdot_{\rm op}$
defined by $a \cdot_{\rm op} b = ba$ for $a,b \in A$.
The next result generalizes \cite[Remark 1.2.4]{nas2004}.

\begin{prop}
Suppose that $A$ is noncommuatively $G$-graded.
Then $A^{\rm op}$ is noncommutatively $G$-graded
with $(A^{\rm op})_g = A_{g^{-1}}$.
Furthermore, the association
$A \mapsto A^{\rm op}$, on objects of $G$-NCGA,
and $f^{\rm op}=f$, on morphisms of $G$-NCGA,
defines an automorphism of the 
category $G$-NCGA.
\end{prop}

\begin{proof}
Take $g,h \in G$. Then 
$( A^{\rm op} )_g ( A^{\rm op} )_h =
A_{g^{-1}} A_{h^{-1}} \subseteq 
A_{g^{-1} h^{-1}} + A_{h^{-1} g^{-1}} =
A_{(hg)^{-1}} + A_{(gh)^{-1}} =
( A^{\rm op} )_{hg} + ( A^{\rm op} )_{gh}$.
Suppose that $f : A \rightarrow B$ is a morphism
in $G$-NCGA. Then
$f^{\rm op} ( ( A^{\rm op} )_g ) = f ( A_{g^{-1}} ) \subseteq
B_{g^{-1}} = ( B^{\rm op} )_g$.
Since $(( A )^{\rm op})^{\rm op} = A$, the last
statement follows.
\end{proof}

\begin{prop}
Let $A$ be noncommutatively $G$-graded and
suppose that $\theta : H \rightarrow G$ is a monomorphism of groups. 
Define the $H$-graded additive group $A_H$ by
$(A_H)_h = A_{\theta(h)}$ for $h \in H$.
Then $A_H = \oplus_{h \in H} (A_H)_h$ is a 
noncommutatively $H$-graded algebra.
The correspondence $A \mapsto A_H$, on objects of $G$-NCGA,
and by restriction $f|_{A_H}$, on morphisms 
$f : A \rightarrow B$ of $G$-NCGA, defines a functor 
$(\cdot)_H$ : $G$-NCGA $\rightarrow$ $H$-NCGA.
\end{prop}

\begin{proof}
Take $h_1 , h_2 \in H$. Then
$( A_H )_{h_1} ( A_H )_{h_2} = 
A_{\theta(h_1)} A_{\theta(h_2)} \subseteq 
A_{\theta(h_1) \theta(h_2)} +   A_{\theta(h_2) \theta(h_1)} =
A_{\theta(h_1 h_2)} + A_{\theta(h_2 h_1)} =
( A_H )_{h_1 h_2} + ( A_H )_{h_2 h_1}$.
The last statement is immediate.
\end{proof}

\begin{prop}
Let $A$ be noncommutatively $H$-graded and
suppose that $\theta : H \rightarrow G$ is a monomorphism of groups.
Define the $G$-graded additive group $\overline{A}$ by
$\overline{A}_g = A_{\theta^{-1}(g)}$, if $g \in \theta(H)$,
and $\overline{A}_g = \{ 0 \}$, if $g \notin \theta(H)$.
Then $\overline{A} = \oplus_{g \in G} \overline{A}_g$ is a 
noncommutatively $G$-graded algebra.
Given a morphism $f : A \rightarrow B$ in $H$-NCGA, 
define the morphism $\overline{f} : \overline{A} \rightarrow \overline{B}$
in $G$-NCGA by the additive extension of the relations
$\overline{f}( a ) = f(a)$, for $a \in \overline{A}_g$
such that $g \in \theta(H)$.
The correspondence $A \mapsto \overline{A}$, on objects of $H$-NCGA,
and $f \mapsto \overline{f}$, on morphisms of $H$-NCGA,
defines a functor $\overline{ ( \cdot ) }$ : $H$-NCGA $\rightarrow$ $G$-NCGA.
\end{prop}

\begin{proof}
Take $g_1,g_2 \in G$. 
Case 1: $g_1 \notin \theta(H)$ or $g_2 \notin \theta(H)$. 
Then $\overline{A}_{g_1} = \{ 0 \}$ or $\overline{A}_{g_2} = \{ 0 \}$ so that 
$\overline{A}_{g_1} \overline{A}_{g_2} = \{ 0 \} \subseteq \overline{A}_{g_1 g_2} + \overline{A}_{g_2 g_1}$.
Case 2: There are $h_1,h_2 \in H$ such that
$\theta(h_1) = g_1$ and $\theta(h_2) = g_2$.
Then $\overline{A}_{g_1} \overline{A}_{g_2} =
A_{\theta^{-1}(g_1)} A_{\theta^{-1}(g_2)} =
A_{h_1} A_{h_2} \subseteq 
A_{h_1 h_2} + A_{h_2 h_1} =
A_{\theta^{-1}(g_1 g_2)} + A_{\theta^{-1}(g_2 g_1)} =
\overline{A}_{g_1 g_2} + \overline{A}_{g_2 g_1}$.

Take $g \in G$ and $h \in H$ such that $\theta(h) = g$.
Suppose that $f : A \rightarrow B$ is a morphism in $H$-NCGA.
Then $f ( \overline{A}_g ) = f ( A_h ) \subseteq B_h = \overline{B}_g$.
If $f' : B \rightarrow C$ is another morphism in $H$-NCGA,
then, clearly, $\overline{f' \circ f} = \overline{f'} \circ \overline{f}$
so that $\overline{ ( \cdot ) }$ is a functor $H$-NCGA $\rightarrow$ $G$-NCGA.
\end{proof}

The next result is a generalization of \cite[Proposition 1.2.1]{nas2004}.

\begin{prop}\label{adjoint}
If  $\theta : H \rightarrow G$ is a monomorphism of groups, then
$( \overline{(\cdot)} , (\cdot)_H ) ) $ is an adjoint pair of functors.
\end{prop}

\begin{proof}
Suppose that $A$ is a noncommutatively $H$-graded algebra and 
that $B$ is a noncommutatively $G$-graded algebra.
Define a map $$\Phi_{A,B} : {\rm hom}_{G-{\rm NCGA}} ( \overline{A} , B ) \rightarrow
{\rm hom}_{H-{\rm NCGA}} ( A , B_H )$$ in the following way.
Given a morphism $f : \overline{A} \rightarrow B$ in $G$-NCGA,
put $\Phi_{A,B}(f) = f_H$. Then $\Phi_{A,B}$ is a bijection.
In fact, define $$\Phi_{A,B}^{-1} : {\rm hom}_{H-{\rm NCGA}} ( A , B_H ) \rightarrow
{\rm hom}_{G-{\rm NCGA}} ( \overline{A} , B )$$ in the following way.
Given a morphism $f' : A \rightarrow B_H$
in $H$-NCGA, put $\Phi_{A,B}^{-1}(f') = \overline{f'}$.
Given a morphism $p : A' \rightarrow A$ in $H$-NCGA and a 
morphism $q : B \rightarrow B'$ in $G$-NCGA, then 
the following diagram is commutative
\begin{equation}\label{THREE}
\CD
{\rm hom}_{G-{\rm NCGA}} ( \overline{A} , B ) @> \Phi_{A,B} >> {\rm hom}_{H-{\rm NCGA}} ( A , B_H ) \\
@V   \epsilon   VV @VV \delta V \\
{\rm hom}_{G-{\rm NCGA}} ( \overline{A'} , B' ) @> \Phi_{A',B'}  >> {\rm hom}_{H-{\rm NCGA}} ( A' , B'_H ) \\
\endCD 
\end{equation}
Here $\epsilon$ and $\delta$ are defined by $\epsilon(f) = q \circ f \circ \overline{p}$
and $\delta(g) = q_H \circ g \circ p$ for 
$f \in {\rm hom}_{G-{\rm NCGA}} ( \overline{A} , B )$ and 
$g \in {\rm hom}_{H-{\rm NCGA}} ( A , B_H )$. 
Thus $( \overline{(\cdot)} , (\cdot)_H ) ) $ is an adjoint pair of functors.
\end{proof}

The next result generalizes the construction preceding \cite[Proposition 1.2.2]{nas2004}

\begin{prop}\label{epi}
Let $A$ be noncommutatively $G$-graded and
suppose that $\pi : G \rightarrow H$ is an epimorphism of groups. 
Define the $H$-graded additive group $A^H$ by
$(A^H)_h = \oplus_{g \in \pi^{-1}(h)} A_g$ for $h \in H$.
Then $A^H = \oplus_{h \in H} (A_H)_h$ is a 
noncommutatively $H$-graded algebra.
The correspondence $A \mapsto A^H$, on objects of $G$-NCGA,
and by the identity, on morphisms of $G$-NCGA, defines a functor 
$(\cdot)^H$ : $G$-NCGA $\rightarrow$ $H$-NCGA.
\end{prop}

\begin{proof}
Take $h_1,h_2 \in H$. Then 
$$( A^H )_{h_1} ( A^H )_{h_2} = 
( \oplus_{g_1 \in \pi^{-1}(h_1)} A_{g_1} ) ( \oplus_{g_2 \in \pi^{-1}(h_2)} A_{g_2} )$$
$$= \sum_{ g_1 \in \pi^{-1}(h_1), \ g_2 \in \pi^{-1}(h_2) } A_{g_1} A_{g_2} \subseteq
\sum_{ g_1 \in \pi^{-1}(h_1), \ g_2 \in \pi^{-1}(h_2) } A_{g_1 g_2} + A_{g_2 g_1}$$
$$\subseteq \sum_{g \in \pi^{-1}(h_1 h_2)} A_g + 
\sum_{g' \in \pi^{-1}(h_2 h_1)} A_{g'} =
( A^H )_{h_1 h_2} + ( A^H )_{h_2 h_1}.$$
Thus $A^H$ is noncommutatively $H$-graded.
The last part is clear.
\end{proof}

\begin{prop}
Suppose that $N$ is a normal subgroup of $G$ and let $A$ be 
an object in $G/N$-NCGA$_{\lambda,\mu}$.
For each $g \in G$, let $F(A)_g$ be the subset $( \oplus_{n \in N} A_{gn} ) g$
of the group ring $A[G]$. Put $F(A) = ( \oplus_{g \in G} F(A)_g , \lambda , \mu )$. 
Then $F(A)$ is a noncommutatively $G$-graded algebra.
The correspondence $A \mapsto F(A)$, on objects of $G/N$-NCGA$_{\lambda,\mu}$,
and by $F (f) (a_{gn} g) = f(a_{gn}) g$, for $g \in G$, $n \in N$ and $a_{gn} \in A_{gn}$,
on morphisms $f$ of $G/N$-NCGA$_{\lambda,\mu}$,
defines a functor $F : G/N$-NCGA$_{\lambda,\mu}$ $\rightarrow$ $G$-NCGA$_{\lambda,\mu}$.
\end{prop}

\begin{proof}
From the proof of \cite[Proposition 1.2.2]{nas2004} it follows that
$\oplus_{g \in G} F(A)_g$ is $G$-graded. Therefore, from the discussion 
in Example \ref{lambdamu} we get that $F(A)$ is noncommutatively $G$-graded
and $F$ defines a functor $G/N$-NCGA$_{\lambda,\mu}$ $\rightarrow$ $G$-NCGA$_{\lambda,\mu}$.
\end{proof}

Now we will show a generalization of \cite[Proposition 1.2.2]{nas2004}
making use of the construction in Example \ref{lambdamu}.

\begin{prop}\label{adjointagain}
Suppose that $N$ is a normal subgroup of $G$ and consider
the canonical epimorphism $\pi : G \rightarrow G/N$.
Let $(\cdot)^{G/N}$ denote the functor
$G$-NCGA$_{\lambda,\mu}$ $\rightarrow$ $H$-NCGA$_{\lambda,\mu}$,
obtained from Proposition \ref{epi} by restriction.
Then $( ( \cdot )^{G/N} , F )$ is an adjoint pair of functors.
\end{prop}

\begin{proof}
Suppose that $A$ is an object in $G$-NCGA$_{\lambda,\mu}$ and 
that $B$ is an object in $G/N$-NCGA$_{\lambda,\mu}$.
Define a map 
$$\Psi_{A,B} : {\rm hom}_{G/N-{\rm NCGA}_{\lambda,\mu}} ( A^{G/N} , B ) \rightarrow
{\rm hom}_{G-{\rm NCGA}_{\lambda,\mu}} ( A , F(B) )$$ in the following way.
Given a morphism $f : A^{G/N} \rightarrow B$ in $G$-NCGA,
put $\Psi_{A,B}(f)(a_g) = f(a_g)g$, for $a_g \in A_g$, and extended biadditively. 
Then $\Psi_{A,B}$ is a bijection (see the proof of \cite[Proposition 1.2.2]{nas2004}).
Now we show that $\Psi_{A,B}(f)$ respects the multiplication $\bullet$.
Take $a_g \in A_g$ and $b_h \in A_h$. Then
$$\Psi_{A,B}(f)( a_g \bullet b_h ) = 
\Psi_{A,B}(f)( \lambda_{g,h} a_g b_h + \mu_{h,g} b_h a_g)$$
$$= f( \lambda_{g,h} a_g b_h ) gh + f (\mu_{h,g} b_h a_g) hg =
\lambda_{g,h} f( a_g b_h ) gh + \mu_{h,g} f( b_h a_g) hg$$
$$=  \lambda_{g,h} f(a_g) f(b_h) gh + \mu_{h,g} f(b_h)f(a_g) hg =
f(a_g)g \bullet f(b_h)h$$
$$= \Psi_{A,B}(f)(a_g) \bullet \Psi_{A,B}(f)(b_h).$$
It also follows from the $G$-graded case (see loc. cit.) 
that the map $\Psi_{A,B}$ is natural in $A$ and $B$.
Therefore, $( ( \cdot )^{G/N} , F )$ is an adjoint pair of functors.
\end{proof}

It is not clear to the author of the present article whether 
the following question can be answered in the affirmative.

\begin{que}
Does the functor $(\cdot)^{G/N} :$ 
$G$-NCGA $\rightarrow$ $H$-NCGA,
obtained from Proposition \ref{epi}, have a right adjoint? 
\end{que}

\section{The Graded Tensor Algebra}\label{gradedtensoralgebra}

In this section, we fix the notation
concerning $G$-graded modules and we state some well known 
results (see Propositions \ref{gradingdirectsum}-\ref{tensor}) 
that will be used in the following section.

Throughout this section, $M$ denotes a left $R$-module.
Recall that $M$ is called $G$-graded if there is a family 
$\{ M_g \}_{g \in G}$ of $R$-submodules of $M$
such that $M = \oplus_{g \in G} M_g$ as $R$-modules.
The next result follows from well-known
properties of direct sums of modules (see e.g. \cite[Chapter III]{lang2002}).

\begin{prop}\label{gradingdirectsum}
If $\{ M^{(i)} \}_{i \in I}$
is a family of $G$-graded modules, then
$\oplus_{i \in I} M^{(i)}$ is a graded module, where, for each $g \in G$,
we put $( \oplus_{i \in I} M^{(i)} )_g =  \oplus_{i \in I} M_g^{(i)}$.
\end{prop}

We will refer to the above grading as the 
canonical direct sum grading.
The next result follows immediately from 
well-known properties of tensor
products of modules (see e.g. \cite[Chapter XVI]{lang2002}).
All tensors are taken over $R$.
Let $\mathbb{N}$ denote the set of non-negative integers.
If $n \in \mathbb{N} \setminus \{ 0 \}$ and $\{ M^{(i)} \}_{i=1}^n$ 
is a family of graded modules, then
we let $\otimes_{i=1}^n M^{(i)}$ denote $M^{(1)} \otimes \cdots \otimes M^{(n)}$.
An element in $\otimes_{i=1}^n M^{(i)}$ of the form
$m_1 \otimes \cdots \otimes m_n$ will be referred to as a monomial.

\begin{prop}\label{gradingtensor}
If $n$ is a positive integer and $\{ M^{(i)} \}_{i=1}^n$ 
is a family of graded modules, then 
$\otimes_{i=1}^n M^{(i)}$ 
is a graded module, where, for each $g \in G$, 
$( \otimes_{i=1}^n M^{(i)} )_g$
is defined to be the submodule of $\otimes_{i=1}^n M^{(i)}$ generated
by all monomials $m_1 \otimes \cdots \otimes m_n$,
where, $m_i \in M_{g_i}^{(i)}$, for $i = 1,\ldots,n$, for some
$g_i \in G$ with the property that $g_1 \cdots g_n = g$.
\end{prop}

Take $n \in \mathbb{N}$ and suppose that $M$ is a
$G$-graded module. If $n=0$, then put $M^{\otimes n} = R$,
where the latter is trivially graded, that is $R_e = R$
and $R_g = \{ 0 \}$, if $g \in G \setminus \{ e \}$.
If $n \geq 1$, then put 
$M^{\otimes n} = M \otimes \cdots \otimes M$
($n$ times) and let $M^{\otimes n}$ be equipped with the grading 
introduced in Proposition \ref{gradingtensor}.
We will refer to this as the canonical tensor grading on $M^{\otimes n}$.
Recall that the tensor algebra $T(M)$
is defined to be the direct sum 
$\oplus_{n \in \mathbb{N}} M^{\otimes n} =
R \oplus M \oplus (M \otimes M) \oplus (M \otimes M \otimes M) \oplus \cdots$ 
as a module. The multiplication in $T(M)$ is indicated by $\otimes$
and is defined on monomials in the following way.
Take $m,n \in \mathbb{N}$. Take monomials 
$x \in M^{\otimes m}$ and $y \in M^{\otimes n}$.
If $m = 0$ (or $n = 0$), then $x \in R$ (or $y \in R$)
and we put $x \otimes y = xy$ (or $x \otimes y = yx$)
as elements in the module $M^{\otimes n}$ (or $M^{\otimes m}$).
If $m,n \geq 1$, then 
$x = v_1 \otimes \cdots \otimes v_m$ and
$y = w_1 \otimes \cdots \otimes w_n$, for some
$v_1,\ldots,v_m,w_1,\ldots,w_n \in M$, and we put
$x \otimes y = 
v_1 \otimes \cdots v_m \otimes w_1 \otimes \cdots \otimes w_n.$
Define the structure of a $G$-graded module on $T(M)$ in the following way.
Let $T(M)$ be equipped with the canonical direct sum grading,
defined, in turn, by the canonical tensor gradings
on $\{ M^{\otimes n} \}_{ n \in \mathbb{N} }$.
In other words, for each $g \in G$, put
$T(M)_g = \oplus_{n \in \mathbb{N}} ( M^{\otimes n} )_g$.
From Proposition \ref{gradingdirectsum} and 
Proposition \ref{gradingtensor} it follows that
$T(M)$ is a graded module. 
We will refer to this grading as the canonical 
grading on $T(M)$.

\begin{prop}\label{tensor}
If $M$ is $G$-graded module, then $T(M)$ is $G$-graded
as an algebra.
\end{prop}

\begin{proof}
Take $g,h \in G$, $x \in T(M)_g$ and $y \in T(M)_h$. We wish to show that
$x \otimes y \in T(M)_{gh}$.
Since this is clear if $m=0$ or $n=0$, we only need to consider
the case when $m,n \geq 1$.
We may assume that $x$ and $y$ are monomials in,
respectively, $(M^{\otimes m})_g$ and $(M^{\otimes n})_h$.
Therefore there are $g_1,\ldots,g_m,h_1,\ldots,h_n \in G$,
$v_i \in M_{g_i}$, for $i = 1,\ldots,m$, and
$w_j \in M_{h_j}$, for $j = 1,\ldots,n$, such that
$g = g_1 \cdots g_m$, $h = h_1 \cdots h_n$,
$x = v_1 \otimes \cdots \otimes v_m$ and
$y = w_1 \otimes \cdots \otimes w_n$.
%We wish to show that $xy \in ( M^{\otimes (m+n)} )_{gh}$.
Then, since
$g_1 g_2 \cdots g_m h_1 h_2 \cdots h_n = gh$, we get that
$xy = v_1 \otimes \cdots \otimes v_m  
\otimes w_1 \otimes \cdots \otimes w_n \in ( M^{\otimes (m+n)} )_{gh}$.
\end{proof}

\section{Noncommutatively graded Lie algebras}\label{gradedliealgebras}

In this section, we study the
particular instance of noncommutatively $G$-graded Lie algebras
(see Definition \ref{defngradedliealgebras}),
we establish the existence of universal graded enveloping algebras 
(see Proposition \ref{lieenveloping}) and we show a 
graded version of the Poincar\'{e}-Birkhoff-Witt theorem
(see Proposition \ref{pbw}).

For the rest of the article, $L$ denotes a Lie algebra.
By this we mean that $L$ is an $R$-algebra which is equipped
with an $R$-bilinear product $[\cdot,\cdot] : L \times L \rightarrow L$
such that for all $a,b,c \in L$ the relations
$[a,a] = 0$ and $[a,[b,c]] + [c,[a,b]] + [b,[c,a]] = 0$ hold.

\begin{defn}\label{defngradedliealgebras}
We say that $L$ is a noncommutatively $G$-graded Lie algebra
if there is a family 
$\{ L_g \}_{g \in G}$ of $R$-submodules of $L$
such that $L = \oplus_{g \in G} L_g$, as $R$-modules, and 
for all $g,h \in G$ the inclusion $[L_g , L_h] \subseteq L_{gh} + L_{hg}$ holds.
We say that the category of noncommutatively $G$-graded Lie algebras,
denoted by $G$-NCGLA, is obtained by taking
noncommutatively $G$-graded Lie algebras as objects and
and for the morphisms between such objects
$L$ and $L'$ we take the Lie algebra homomorphisms
$f : L \rightarrow L'$ satisfying $f(L_g) \subseteq L_g'$, for $g \in G$. 
\end{defn}

\begin{rem}
If we let $G$-GAA denote the subcategory of $G$-GA having
associative $G$-graded algebras as objects, then 
the commutator (\ref{defcommutator}) defines a functor
Lie : $G$-GAA $\rightarrow$ $G$-NCGLA.
\end{rem}

Now we extend the definition of a universal enveloping 
algebra (see e.g. \cite[Chapter V]{jacobson1979}) to the noncommutatively graded situation.

\begin{defn}
Let $L$ be a noncommutatively $G$-graded Lie algebra.
A pair $( U , i )$, where $U$ is an 
associative unital $G$-graded algebra and
$i : L \rightarrow \mbox{Lie}(U)$
is a morphism in $G$-NCGLA,
is called a universal $G$-graded enveloping algebra
if the following holds:
if $A$ is any associative unital $G$-graded algebra
and $j : L \rightarrow \mbox{Lie}(A)$ is a morphism in $G$-NCGLA,
then there exists a unique morphism 
$k : \mbox{Lie}(U) \rightarrow \mbox{Lie}(A)$
in $G$-NCGLA such that $j = k \circ i$.
\end{defn}

\begin{prop}\label{lieenveloping}
Every noncommutatively $G$-graded Lie algebra
which is free with a homogeneous basis 
has a universal $G$-graded enveloping algebra.
\end{prop}

\begin{proof}
We proceed exactly as in classical ungraded case (see e.g. \cite[Chapter V]{jacobson1979}).
Let $I$ be the ideal of $T(L)$
generated by all elements of the form
\begin{equation}\label{first}
[a,b] - a \otimes b + b \otimes a
\end{equation}
for $a,b \in L$ and put $U = T(L) / I.$
Let $X$ denote the set of all elements 
of the form (\ref{first}) with $a$ and $b$ homogeneous.
Then $I = T(L) X T(L)$, so $I$ is a $G$-graded ideal.
Therefore $U$ is a unital associative $G$-graded algebra.
Let $l : L \rightarrow T(L)$ denote the inclusion and let
$q : T(L) \rightarrow U$ denote the quotient map.
Let $i : L \rightarrow U$ be the composition of the graded maps
$l : L \rightarrow T(L)$ and $q : T(L) \rightarrow U$.
Let $B = \{ b_v \}_{v \in V}$ be a a homogeneous basis for $L$.
Suppose that $A$ is a unital graded algebra and there
is a graded homomorphism $j : L \rightarrow {\rm Lie}(A).$
By ungraded universality of $(U,i)$ there is a unique
homomorphism $k : U \rightarrow A$ such that
$j = k \circ i$. The map $k$ is defined by
$k(1) = 1$ and 
$k( b_{v_1} \otimes \cdots \otimes b_{v_n} ) = 
j( b_{v_1} ) \otimes \cdots \otimes j( b_{v_n} ).$
Since $j$ is graded, $k$ is also graded.
\end{proof}

\begin{cor}
Every noncommutatively $G$-graded Lie algebra
over a field has a universal graded enveloping algebra.
\end{cor}

\begin{proof}
Suppose that $R$ is a field. For each $g \in G$
choose a basis $B_g$ for $L_g$. 
It is clear that $B = \cup_{g \in G} B_g$ is a homogogeneous
basis for $L$. The claim now follows directly from Proposition \ref{lieenveloping}.
\end{proof}

The following result is a graded analogue of the 
Poincar\'{e}-Birkhoff-Witt theorem.

\begin{prop}\label{pbw}
Suppose that $L$ is a noncommutatively $G$-graded Lie algebra
which is free with a homogeneous basis $B = \{ b_v \}_{v \in V}$
where the set $V$ is equipped with a total order $\leq$.
Then the cosets
$1 + I$ and
$b_{v_1} \otimes \cdots \otimes b_{v_n} + I,$
where $v_1 < \cdots < v_n$,
form a homogeneous basis for $U$.
\end{prop}

\begin{proof}
This follows from the ungraded Poincar\'{e}-Birkhoff-Witt 
theorem (see e.g. \cite[Chapter V]{jacobson1979}).
\end{proof}

\end{document}